\documentclass[10pt]{amsart}
\usepackage[T1]{fontenc}
\usepackage[english]{babel}
\usepackage{xspace}
\usepackage{etoolbox} 
\patchcmd{\thebibliography}{\list}{\printbibepigraph\list}{}{}
\newcommand{\bibepigraph}[1]{%
  \def\printbibepigraph{\begin{flushright}#1\end{flushright}}%
}
\def\printbibepigraph{} 

\usepackage
{geometry}
\usepackage{kpfonts}

\usepackage{amsmath,amssymb,amsthm,enumerate,xspace}
\usepackage{accents}
\usepackage[colorlinks=true,citecolor=blue,linkcolor=black]{hyperref} 
\usepackage{mathrsfs}

\numberwithin{equation}{section}

\newtheorem{thm}{Theorem}
\newtheorem*{thm*}{Theorem}

\newtheorem*{prop*}{Proposition}

\newtheorem{cor}[thm]{Corollary}
\newtheorem*{cor*}{Corollary}

\newtheorem*{iproblem*}{Problem}

\theoremstyle{definition}

\newtheorem*{defi*}{Definition}

\newtheorem*{rem*}{Remark}
\newtheorem*{warn*}{Warning}
\newtheorem*{com*}{Comment}

\newcommand{\sA}{\mathscr{A}}

\newcommand{\RR}{\mathbf{R}}

\newcommand{\CC}{\mathbf{C}}

\newcommand{\Isom}{{\rm Isom}}
\newcommand{\Aut}{{\rm Aut}}

\newcommand{\inv}{^{-1}}
\newcommand{\ro}{\varrho}
\newcommand{\fhi}{\varphi}

\newcommand{\epsi}{\epsilon}

\newcommand{\cato}{{\upshape CAT(0)}\xspace}
\newcommand{\one}{\boldsymbol{1}}
\newcommand{\bnda}{\Delta}
\newcommand{\bnd}{\partial}

\newcommand{\prob}{\mathscr{P}} 

\newcommand{\ru}{\mathrm{C}^\mathrm{b}_\mathrm{ru}}
\newcommand{\cont}{\mathrm{C}}

\newcommand{\meas}{\mathscr{M}}
\newcommand{\measb}{\meas^\mathrm{b}}
\newcommand{\bo}{P}
\newcommand{\Mod}{\nabla}

\title[Gelfand pairs admit an Iwasawa decomposition]{Gelfand pairs admit an Iwasawa decomposition}

\author[Nicolas Monod]{Nicolas Monod}
\address{EPFL, Switzerland}
\begin{document}
\begin{abstract}
Every Gelfand pair $(G,K)$ admits a decomposition $G=K\bo$, where $\bo<G$ is an amenable subgroup. In particular, the Furstenberg boundary of $G$ is homogeneous.

Applications include the complete classification of non-positively curved Gelfand pairs, relying on earlier joint work with Caprace, as well as a canonical family of pure spherical functions in the sense of Gelfand--Godement for general Gelfand pairs.
%
\end{abstract}
\maketitle

\thispagestyle{empty}
%

Let $G$ be a locally compact group. The space $\measb(G)$ of bounded 
measures on $G$ is an algebra for convolution, which is simply the push-forward of the multiplication map $G\times G \to G$.

\begin{defi*}
Let $K<G$ be a compact subgroup. The pair $(G,K)$ is a \textbf{Gelfand pair} if the algebra $\measb(G)^{K,K}$ of bi-$K$-invariant measures is commutative.
\end{defi*}

This definition, rooted in Gelfand's 1950 work~\cite{Gelfand50_short}, is often given in terms of algebras of \emph{functions}~\cite{Faraut83}
. This is equivalent, by an approximation argument in the narrow topology
, but has the inelegance of requiring the choice (and existence) of a Haar measure on $G$.

Examples of Gelfand pairs include notably all connected semi-simple Lie groups $G$ with finite center, where $K$ is a maximal compact subgroup. Other examples are provided by their analogues over local fields~\cite{Gross91}, and non-linear examples include automorphism groups of trees~\cite{Olshanskii77_en},\cite{Amann_PhD}.

All these ``classical'' examples also have in common another very useful property: they admit a \emph{co-compact amenable subgroup} $P<G$. In the semi-simple case, $P$ is a minimal parabolic subgroup. Moreover, the Iwasawa decomposition implies that $G$ can be written as $G=KP$. This note shows that this situation is not a coincidence:

\begin{thm*}
Let $(G,K)$ be a Gelfand pair. Then $G$ admits a co-compact amenable subgroup $\bo < G$ such that $G=K\bo$.
\end{thm*}

The mere existence of $\bo$ has a number of strong consequences discussed below. Most immediate is that $G$ belongs to the exclusive club whose members boast a homogeneous Furstenberg boundary:

\begin{cor}\label{cor:max}
Choose a \emph{maximal} subgroup $\bo < G$ as in the Theorem.

Then the Furstenberg boundary of $G$ is the homogeneous space $\bnd G = G/\bo \cong K/(K\cap\bo)$.

In particular, $\bo$ is unique up to conjugacy.
\end{cor}

\noindent
Another general consequence is that $G$ is \emph{exact} in the sense of C*-algebras~\cite[\S7.1]{Kirchberg-Wassermann99}.

\begin{rem*}
The proof of the Theorem is easy. What surprises us (besides the fact that it went unnoticed during decades of harmonic analysis on Gelfand pairs) is that the unique group $\bo$ of Corollary~\ref{cor:max} is obtained by purely existential methods. Indeed, the author is unaware of a constructive proof --- \emph{or even of a heuristic based on the classical Iwasawa decomposition}, explaining $(G,K)\mapsto \bo$.
\end{rem*}

We next derive a more geometric illustration of how consequential the existence of $\bo$ is. The above classical examples of Gelfand pairs are all \textbf{\cato groups} in the sense that they occur as cocompact isometry groups of non-positively curved spaces: either \emph{Riemannian symmetric spaces} or \emph{Euclidean buildings}. General \cato groups constitute a much more cosmopolitan category populated by all sorts of exotic spaces hailing from combinatorial group theory, Kac--Moody theory, etc. Using the ``indiscrete Bieberbach theorem'' established with P.-E.~Caprace~\cite{Caprace-Monod_bib}, the Theorem of this note leads to a complete classification of \cato Gelfand pairs:

\begin{cor}\label{cor:NPC}
Let $(G,K)$ be a Gelfand pair and assume that $G<\Isom(X)$ acts co-compactly on a geodesically complete locally compact \cato space $X$.

Then $X$ is a product of Euclidean spaces, Riemannian symmetric spaces of non-compact type, Bruhat--Tits buildings and biregular trees.

In particular, $G$ lies in a product of Gelfand pairs belonging to the classical sets of examples above.
\end{cor}

\noindent
This statement contains for instance a result by Caprace--Ciobotaru~\cite{Caprace-Ciobotaru}, namely: let $X$ be an irreducible locally finite thick Euclidean building. If $G=\Aut(X)$ (or any co-compact subgroup $G<\Isom(X)$) is a Gelfand pair for some compact $K<G$, then $X$ is Bruhat--Tits.

Similarly, the statement contains some cases of results by Abramenko--Parkinson--Van Maldeghem~\cite{Abramenko-Parkinson-VanMaldeghem} and L\'ecureux~\cite[\S7]{Lecureux_PhD},\cite{Lecureux10} establishing the non-commutativity of Hecke algebras associated to certain Coxeter groups. Namely, when Kac--Moody theory associates to them a locally finite thick building, Corollary~\ref{cor:NPC} implies that the Hecke algebra can only be commutative in the affine case.

\medskip

The ``Iwasawa decomposition'' $G=K\bo$ is stronger yet than the existence of $\bo$. For instance, it is a key ingredient for results of Furman~\cite[Thm.~10]{Furman03} and it could shed some light on the \emph{spherical dual} of $G$, see below. It should also impose further restrictions on the centraliser lattice in case $G$ is a compactly generated simple group, see~\cite{Caprace-Reid-Willis_2017_II}. Already the existence of $\bo$ implies that this lattice is at most countable: see~\cite[pp.~11--12]{Caprace-Reid-Willis_2017_II} and use that $G/\bo$ is metrisable in this setting.

\medskip
We now contemplate some of the analytic legacy that the decomposition $G=K\bo$ bestows upon a general Gelfand pair $(G,K)$. Following Gelfand and Godement~\cite{Godement57}, the fundamental building block of non-commutative Fourier--Plancherel theory is given by positive definite \textbf{spherical functions} on Gelfand pairs, namely continuous $\fhi\colon G\to \CC$ satisfying
\begin{equation*}
\fhi(x)\, \fhi(y) = \int_K  \fhi(x k y) \, d k \kern5mm\forall x,y\in G
\end{equation*}
where the integration is with respect to the unique Haar probability measure on $K$; see also~\cite{Dieudonne79} and~\cite{Wolf07_short}. This is the abstract generalisation of addition formulas for special functions such as Legendre functions~\cite{Vilenkin_short}.

Here is how $\bo$ enters the picture:

Let $\Mod_\bo$ be the modular function of $\bo$, which is non-trivial unless $G$ itself is amenable and $G=\bo$. Then $\ro(kp) = \Mod_\bo(p)$ gives a well-defined continuous function $\ro\colon  G\to\RR_{>0}$ when $k\in K$, $p\in \bo$ because $\Mod_\bo$ vanishes on $K\cap\bo$. For every parameter $s\in \CC$, define
\begin{equation*}
\fhi_s(g) = \int_K \ro(g\inv k)^{\frac12 + i s} \, d k.
\end{equation*}
In view of Corollary~\ref{cor:max}, \emph{$\fhi_s$ is actually canonically attached to the pair $(G,K)$} up to conjugation. On the other hand, $\fhi_s$ is the matrix coefficient of the (projectively) unique $K$-fixed vector in a parabolically induced representation from $P$. In particular, $\fhi_s$ is a pure positive definite spherical function on $G$ for each real $s$.

This is classical for semi-simple groups, where $\fhi_s$ above is the \textbf{Harish-Chandra formula}; the Theorem makes it available for general Gelfand pairs, as desired by Godement~\cite[\S16]{Godement52}. Of course this only gives a principal series and suggests to investigate fully the characters of $P$.


\medskip
\emph{Proof of the Theorem and of Corollary~\ref{cor:max}.}
We recall that an \textbf{affine $G$-flow} is a non-empty compact convex set $C$ in some locally convex topological vector space over $\RR$, endowed with a jointly continuous $G$-action preserving the affine structure of $C$. An affine flow is called \textbf{irreducible} if it does not contain any proper affine subflow. An argument due to Furstenberg implies that $G$ admits an irreducible flow $\bnda G$ which is \textbf{universal} in the sense that it maps onto every irreducible flow. Moreover, $\bnda G$ is unique up to unique isomorphisms. It turns out that $\bnda G$ is the simplex of probability measures $\prob(\bnd G)$ over the Furstenberg boundary $\bnd G$ of $G$, and that this is actually one of the possible \emph{definitions} of $\bnd G$. For all this, we refer to~\cite{Glasner_LNM_short}.

We shall be more interested in the convex subset $\prob(G)$ of $\measb(G)$ consisting of the probability measures, as well as in the corresponding subset $\prob(G)^{K,K}$. We note the following straightforward facts:

\begin{itemize}
\item $\prob(G)$ is closed under the multiplication given by convolution.
\item The monoid $\prob(G)$ contains $G$ via the identification of points with Dirac masses.
\item The normalised Haar measure $\kappa$ of $K$ is an idempotent belonging to $\prob(G)^{K,K}$.
\item $\prob(G)^{K,K} = \kappa \prob(G) \kappa$; it is a monoid with $\kappa$ as identity.
\end{itemize}

By generalised vector-valued integration~\cite[IV\S7.1]{BourbakiINT14}, any affine $G$-flow $C$ is endowed with an action of the monoid $\prob(G)$ which is affine in both variables. It will be crucial below that this action is moreover continuous for the variable in $C$. One way to see this is to check first that any $\mu\in \prob(G)$ induces a continuous map $C\to \prob(C)$ by push-forward on orbits, using that the $G$-action on $C$ is equicontinuous over compact subsets of $G$. Then observe that the action of $\mu$ is obtained by composing this  map $C\to \prob(C)$ with the continuous barycenter map $\prob(C)\to C$.

Since $K$ is compact, it has a non-empty fixed-point set $C^K$; better yet, the idempotent $\kappa$ provides a continuous projection $\kappa\colon C\to C^K$. In particular, the monoid $\sA=\kappa \prob(G) \kappa$ preserves the convex compact set $C^K$.

Only now do we use the assumption that we have a Gelfand pair: the monoid $\sA$ is commutative. Since $\sA$ acts by continuous operators, the Markov--Kakutani theorem therefore implies that $\sA$ fixes a point $p$ in $C^K$. From now on, we assume that $C$ is irreducible. The convex set $\prob(G)p$ is $G$-invariant and hence must be dense. It follows that $\kappa\prob(G)p$ is dense in $C^K$, but $\kappa\prob(G)p$ is $\sA p$ which is reduced to $p$. In conclusion, we have shown that $K$ has a unique fixed point in $C$.

We now apply this to the case where $C=\bnda G$ is the simplex of probability measures on $\bnd G$ and deduce that $K$ fixes a unique such measure on $\bnd G$. Since $K$ is compact, every $K$-orbit supports an invariant measure: the push-forward of $\kappa$. This implies that $K$ has a single orbit in $\bnd G$. In particular, $\bnd G= G/\bo$ for some co-compact subgroup $\bo < G$ and moreover $G= K\bo$.

Next, we observe that $\bo$ is \textbf{relatively amenable} in $G$, which means by definition that every affine $G$-flow has a $\bo$-fixed point. Indeed, this property characterises the subgroups that fix a point in $\bnda G$: this follows from the universal property of $\bnda G$. This characterisation also implies that this $\bo$ is already \emph{maximal} relatively amenable. Indeed, if $\bo'<G$ is relatively amenable and contains $\bo$, it also fixes a point in $\bnda G$; this induces an affine $G$-map $\prob(G/\bo') \to  \prob(G/\bo)$, which must be the identity by universality of $\prob(G/\bo)=\bnda G$.

We recall that relative amenability is equivalent to amenability in a wide class of ambient locally compact groups $G$ including all exact groups, but it is only \emph{a posteriori} that the Theorem implies that $G$ is exact, see~\cite[\S7.1]{Kirchberg-Wassermann99}. In the locally compact setting, it is still an open question to exhibit an example where the weaker relative notion does not coincide with amenability~\cite{Caprace-Monod_rel}. In the co-compact case, however, we can settle the question with the Proposition below and conclude that $\bo$ is amenable. Thus the Proposition will complete the proof.

The following statement is a very basic case of much more general results by Andy Zucker~\cite[Thm.~7.5]{Zucker_prep2019}; the elementary proof below is inspired by reading his preprint.

\begin{prop*}
Let $G$ be a Hausdorff topological group and $\bo<G$ a closed subgroup such that $\bnd G = G/\bo$. Then $\bo$ is amenable.
\end{prop*}

\begin{warn*}
A subgroup of $G$ fixing a point in $\bnd G$ is not necessarily amenable. However, in the locally compact case and assuming $\bnd G$ homogeneous, this follows from the Proposition because amenability of locally compact groups passes to subgroups.
\end{warn*}



\begin{proof}[Proof of the Proposition]
We know that $\bo$ is co-compact and relatively amenable. The latter is equivalent to the existence of a $\bo$-invariant mean $\mu$ on the space $\ru(G)$ of right uniformly continuous bounded functions (cf.\ Thm.~5 in~\cite{Caprace-Monod_rel}). It suffices to show that $\mu$ descends to $\ru(\bo)$, viewed as a quotient of $\ru(G)$ under restriction (by Katetov extension~\cite{Katetov51}). Let thus $f\in \ru(G)$ be any map vanishing on $\bo$; we need to show $\mu(f)=0$ and can assume $f\geq 0$. Given $\epsi>0$ there is an identity neighbourhood $U$ in $G$ such that $f\leq \epsi$ on $U \bo$. By Urysohn's lemma in $G/\bo$, there is $h\in \cont(G/\bo)$ vanishing on a neighbourhood of $\bo$ but taking constant value $\|f\|_\infty$ outside $U \bo$. Viewing $h$ as an element of $\ru(G)$, we thus have $f \leq \epsi \one_G + h$. We now claim $\mu(h)=0$, which finishes the proof since $\epsi$ is arbitrary. The claim follows from the fact that $\mu$ is mapped to a $\bo$-invariant probability measure on $G/\bo$ under the inclusion of $\cont(G/\bo)$ in $\ru(G)$. Indeed, the only $\bo$-invariant probability measure on $G/\bo \cong \bnd G$ is the Dirac mass at $\bo$ by strong proximality of $P$ on $\bnd G$, see~\cite[II.3.1]{Glasner_LNM_short}.
\end{proof}

\begin{proof}[Proof of Corollary~\ref{cor:NPC}]
Consider $G<\Isom(X)$ as in the statement. We first recall that $X$ is \emph{minimal} in the sense that it does not contain a closed convex $G$-invariant proper subset, see~\cite[3.13]{Caprace-Monod_structure}. Next, we recall that general splitting results (1.9 together with~1.5(iii) in~\cite{Caprace-Monod_structure}) allow us to reduce to the case where $X$ has no Euclidean factor. In any Gelfand pair, $G$ is unimodular~\cite[24.8.1]{Simonnet_short}; this, together with the elements collected thus far, allows us to apply Theorem~M in~\cite{Caprace-Monod_amenis_MA}. That result states that $G$ has no fixed point at infinity. On the other hand, our Theorem above provides a subgroup $\bo<\Isom(X)$ acting co-compactly on $X$. We are now in position to apply the indiscrete Bieberbach theorem~\cite[Thm.~B]{Caprace-Monod_bib}, which identifies $X$ with a product of classical spaces as desired.
\end{proof}

We now justify our claims concerning the functions $\fhi_s$ on $G$. Since $G$ is unimodular (reference above), Weil's integration formula~\cite[VII\S2.5]{NBourbakiINT78_bis} implies that the push-forward of $\kappa$ on $G/\bo$ has a Radon--Nikod\'ym cocycle given at $(g, x \bo)$ by $\ro(g\inv x) /\ro(x)$. Therefore, the unitary induction $\pi_s$ of the character $\Mod_\bo^{i s}$ is given on various spaces of functions $f$ on $G/\bo$ by
\begin{equation*}
(\pi_s(g) f) (x\bo) = f(g\inv x\bo) \left(\frac{\ro(g\inv x)}{\ro(x)}\right) ^{\frac12 + i s} .
\end{equation*}
The only $K$-invariant vectors $v$ are constant functions on $G/\bo$ and hence the associated matrix coefficient $\fhi_s(g) = \langle \pi_s(g) v, v\rangle$ is uniquely defined once $v$ has unit norm. The fact that $\fhi_s$ is pure and spherical (for $s\in\RR$) now follows from the general theory of Gelfand pairs, specifically~I.II.6 and~I.III.2 in~\cite{Faraut83}.

\begin{rem*}
A part of the proof of the Theorem is reminiscent of the fact that any irreducible \emph{unitary} representation of $G$ has at most a one-dimensional subspace of $K$-fixed vectors, a fact that actually characterises Gelfand pairs. We recall that the corresponding statement fails for \emph{real} Hilbert spaces, whereas our affine flows are always over the reals.
\end{rem*}

\textbf{Acknowledgements.}
I am grateful to Andy Zucker for sending me his preprint and to Pierre-Emmanuel Caprace for several insightful comments on a preliminary version.





\bibepigraph{\emph{Selberg ne fait aucune esp{\`e}ce d'allusion {\`a} l'existence possible d'une litt{\'e}rature math{\'e}matique.}\\ --- R. Godement~\cite{Godement57}, 1957.}
\bibliographystyle{abbrv}
\bibliography{../BIB/ma_bib}

\end{document}